\documentclass[12pt]{amsart}

\usepackage[bottom=3.5cm,top=3.5cm,left=3cm,right=3cm]{geometry}
\usepackage{amsmath,amsthm,amsfonts,amssymb}
\usepackage{verbatim}
\usepackage{graphicx}
\usepackage{epstopdf}
\usepackage[square,sort,comma,numbers]{natbib}
\usepackage{url}

\newtheorem{lemma}{Lemma}
\newtheorem{theorem}{Theorem}
\newtheorem{corollary}{Corollary}
\newtheorem{proposition}{Proposition}
\newtheorem{definition}{Definition}

\newcommand{\pn}{P^{(n)}}
\newcommand{\hair}[2]{#1^{(+ #2)}}
\newcommand{\ball}{\mathcal{B}_n}
\newcommand{\oneik}{i=\overline{1,k}}
\newcommand{\inv}[1]{\frac{1}{#1}}
\newcommand{\pfx}{\mathcal{P}_f(X)}
\newcommand{\pfxf}{\pfx \rtimes F}
\newcommand{\suml}{\sum\limits}
\newcommand{\sumr}[3]{\sum\limits_{#1=#2}^{#3}}
\newcommand{\obj}{\textit{object}}
\newcommand{\objs}{\textit{objects}}
\newcommand{\lampprod}{\bigoplus\limits_X A \rtimes G}

\title[Superharmonic functions on the Lamplighter graph]{Superharmonic functions on the Lamplighter graph of Thompson's group $F$}

\author{Maksym Chornyi}\thanks{Graduate student at Northwestern University, supervised by Kate Juschenko.}

\begin{document}

\begin{abstract}
	The goal is to extend a non-standard amenability test for groups, based on random walks and superharmonic functions, to group actions on sets, and to apply it to Thompson's group $F$ using certain properties of extensive amenability. While no conclusive answer regarding the amenability of $F$ is given, the approach is helpful in developing a new potentially useful criterion and testing it on a significant subclass of superharmonic functions.
\end{abstract}

\maketitle

\section{Introduction and definitions}

Amenability of discrete groups has a strong connection to the notion of random walks.

Let $\mu$ be a measure on a group $G$. This measure can be seen as the source of a left random walk on $G$, where the transition probabilities are defined by $p(g,h)=\mu(hg^{-1})$. 

Similarly, if a discrete group $G$ acts on a countable set $X$, the induced random walk is given by transition probabilities $p(x,y)= \sum\limits_{gx=y} \mu(g) $. The summation sign is needed in case several group elements map $x$ to $y$.

\begin{definition}
	Let $\mu$ and $\nu$ be two measures on a discrete group $G$. The convolution of $\mu$ and $\nu$ is defined by $(\mu*\nu)(g) = \sum\limits_{h \in G} \mu(h)\nu(h^{-1}g) $.
\end{definition}

A convolution of two probability measures is always a probability measure, and the operation of taking convolutions is associative in the sense $\mu * (\nu * \rho) = (\mu * \nu) * \rho.$

The {\it $n$-th convolution} of measure $\mu$ is defined by $\mu^{*n} = \underbrace{\mu * \ldots * \mu}_{n \text{ times}}$.

\begin{definition}
	A measure $\mu$ on a group $G$ is called finitely supported if there is a finite set $S \subset G$ such that $\mu(S)=1$.
\end{definition}

\begin{definition}
	A measure $\mu$ on a group $G$ is called generating if $G = \bigcup\limits_{n \ge 0} \emph{supp}(\mu^{*n})$.
\end{definition}

\begin{definition}
	A measure $\mu$ on a group $G$ is called symmetric if $\mu(A) = \mu(A^{-1})$ for any $A \subset G$.
\end{definition}

For a given random walk on a countable set $X$, denote $P_n(x,y) = \sum\limits_{gx=y} \mu^{*n} (g)$, which represents the probability of reaching point $y$ from point $x$ in exactly $n$ steps.

\textbf{Standing assumption.} Unless specified otherwise, further on in this dissertation we assume that all described group actions are transitive and the corresponding measures are finitely supported, generating and symmetric.

\begin{definition}
	Green's function of points $x$ and $p$ of a random walk with a parameter $z$ is given by $G(x,p|z) = \sum \limits _{n=0}^{\infty} P_n(x,p)z^n$.
\end{definition}

For $z=1$, Green's function represents the expected number of times we visit $p$ when starting from $x$. For $z<1$, we can think of Green's function as the same expected number of visits but with the condition that before each step the walk is terminated with probability $z$.

Let $x \in X$ be an arbitrary point. The radius of convergence of the random walk is given by 
\[r(X,\mu)=\rho(X,\mu)^{-1} = \left(\lim\limits_{n\rightarrow \infty} \sqrt[n]{P_n(x,x)}\right)^{-1}\].

According to basic facts from random walk theory (\cite{woess}, Chapter 1), the limit above is always defined and does not depend on the point $x$, even though it can depend on the precise generating set. We call $\rho(X,\mu)$ the {\it inverse spectral radius} of the random walk.

\begin{definition}
	The linear operator $P : l^2(X) \rightarrow l^2(X)$ defined by  $Pf(x) = \linebreak = \sum\limits_g \mu(g) f(gx)$ is called the Markov operator with respect to $\mu$.
\end{definition}

Informally, the Markov operator represents the mathematical expectation of the value of $f$ on the point reached after the first step defined by $\mu$.

\begin{definition}\label{crit}
	A function $f:X \rightarrow \mathbb{R}$ is called $\mathit{superharmonic}$ with respect to a probability measure $\mu$ if $f(x) \ge Pf(x)$ for all $x \in X$. If $f(x)=Pf(x)$ for all $x$, the function is called harmonic.
\end{definition}

It turns out that there is a very explicit connection between random walks and amenability criteria. It is summarized by Kesten's amenability test:

\begin{theorem}\label{kesten}
	Let a group $G$ act on a countable set $X$ and let $\mu$ be a measure on $G$. The following are equivalent:
	
	\begin{enumerate}
		\item The action of $G$ on $X$ is amenable.
		\item The Markov operator corresponding to the random walk defined by $\mu$ has norm $||P||=1$.  
		\item The inverse spectral radius $\rho (X, \mu)$ of the random walk is equal to 1.
	\end{enumerate}
	
\end{theorem}

The proof is based on the spectral theorem for self-adjoint operators and can be found in \cite{woess} (Chapters 10 and 12). Another proof of Kesten's criterion for the left action of $G$ on itself (i.e. the case of amenable groups) can be found in \cite{J-book}.

The following amenability criterion was proved by Sam Northshield in his 1993 paper \cite{northshield}:

\begin{theorem}
	Let $G$ be a countable group and $\mu$ be a measure defined on $G$. Then $G$ is amenable if and only if for any positive superharmonic function $f$ there exists a sequence $\{x_n\}$ in $G$ such that for any $z \in G$, $\frac{f(zx_n)}{f(x_n)} \rightarrow 1$, $n \rightarrow \infty$.
\end{theorem}

We extend the statement to all amenable actions of a discrete group $G$ on a graph $X$. The structure of our proof is based on the original.

The following fact is proven in \cite{woess} (Lemma 7.2, p. 81):

\begin{lemma}
	Let $\mu$ be a generating and symmetric measure on $G$, and $P$ be the corresponding Markov operator. Then the spectral radius of $P$ satisfies $r(P)^{-1} = \inf\{ \lambda : \exists f > 0: \ Pf \le \lambda f \}$.
\end{lemma}

\begin{corollary}
	The action of $G$ on $X$ is amenable if and only if $\inf\{ \lambda : \exists f > 0: \ Pf \le \lambda f \} = 1$.
\end{corollary}

\begin{theorem}
	The action of $G$ on $X$ is amenable if and only if for all positive superharmonic functions $f$ with respect to $\mu$ there exists a sequence $\{E_n\}$ in $X$ such that for all $x \in G$, $\frac{f(xE_n)}{f(E_n)} \rightarrow 1$ as $n \rightarrow \infty$.
\end{theorem}

\begin{proof}
	Let $G$ act amenably on $X$ and let $f$ be a positive superharmonic function. Define $P^{(n)}$ to be the $n$-fold convolution of $P$ with itself and suppose that $\sup\limits_E \frac{\pn f(E)} {f(E)} \neq 1$. Then there exists $\varepsilon \in (0,1) $ such that $\pn f \le \varepsilon^n f$. Let $g = \sum\limits_{0 \le i \le n-1} \frac{P^{(i)}f}{\varepsilon^i}$. Then $Pg - \varepsilon g = \frac{\pn f(E)}{\varepsilon^{n-1}} - \varepsilon f \le 0 $, which is, by Lemma \ref{one}, a contradiction. Hence $\sup\limits_E \frac{\pn f(E)} {f(E)} = 1$.
	
	For $E \in X$, define $f_E(x)=\sqrt{\frac{f(xE)}{f(E)}}$, where $x \in G$. $f_E^2$ is positive superharmonic, since \[ \sum\limits_g \mu(g) \frac{f(gxE)}{f(E)} = \frac{1}{f(E)} \sum\limits_g  \mu(g) f(g \cdot xE) =  \frac{Pf(xE)}{f(E)} \le \frac{f(xE)}{f(E)}. \] 
	
	Since $t \rightarrow t^{1/2}$ is an increasing concave function, by Jensen's inequality $f_E$ is also positive and superharmonic. Note that
	
	\begin{multline*}
	\pn f_E(e) = \sum\limits_x \mu^{(n)} (x) f_E(x) = \sum\limits_x \mu^{(n)} (x) \sqrt{\frac{f(xE)}{f(E)}} = \\ = \sum\limits_F \pn (E,F) \sqrt{\frac{f(F)}{f(E)}} = \frac{\pn [f^{1/2}](E)}{f(E)^{1/2}} .
	\end{multline*}
	
	Thus, $\sup\limits_E \pn f_E(e) = 1$.
	
	Since, for all $E$, $\pn f_E(e)$ is decreasing as a function of $n$, we can choose $E(k)$ such that, for all $n$, $\pn f_{E(k)}(e) \rightarrow 1$ as $k$ goes to infinity. Then we have: 
	
	$1 = \left( f_{E(k)}(e) \right)^2 \ge \pn [f_{E(k)}]^2(e) \ge [\pn f_{E(k)}(e)]^2 \rightarrow 1$. Thus:
	
	$\sum\limits_x \mu^{(n)} (x) \left(f_{E(k)}(x)-1\right)^2 = \pn [f_{E(k)}]^2(e) - 2 \pn f_{E(k)}(e)+1 \rightarrow 0$, $k \rightarrow \infty$. Since $\mu$ is generating, we have $f_{E(k)} \rightarrow 1$ pointwise.
	
	To prove the opposite direction, assume that for every positive superharmonic function $f$ there is a sequence $\{E_n\}$ satisfying the condition of the theorem. Then $\frac{Pf(E_n)}{f(E_n)} = \sum\limits_x \mu(x) \frac{f(xE_n)}{f(E_n)} \rightarrow 1$, which implies amenability by Lemma \ref{one}.
	
\end{proof}

This criterion can be rewritten in terms of words of certain length. We note that each element $g \in G$ can be represented by a word $w(g) \in S^*$, where $S$ is the support of $\mu$ and hence a generating set by convention. Abusing notation, we will normally substitute $g$ for $w(g)$ whenever no confusion arises. 

\begin{corollary}\label{north-equiv}
	Let $\beta: \mathbb{N} \rightarrow \mathbb{R}_{>0}$ be a non-increasing function satisfying $\lim\limits_{n \rightarrow \infty} \beta(n) = 0$.
	The action of $G$ on $X$ is amenable if and only if for any positive superharmonic function $f$ with respect to $\mu$ there exists a sequence $\{E_n\}$ such that any word $g$ of length not exceeding $n$ satisfies $\left|\frac{f(gE_n)}{f(E_n)}-1\right| < \beta(n)$.
\end{corollary}

\begin{proof}
	If there is a sequence $\{E_n\}$ satisfying the conditions of Theorem \ref{north-strong} and $k$ is a positive integer, then $\left|\frac{f(gE_n)}{f(E_n)}-1\right| \rightarrow 0$ for any $g$ of length at most $k$. Since there are only finitely many words of length at most $k$, this implies
	\[\max\limits_{|g| \le k}  \left|\frac{f(gE_n)}{f(E_n)}-1\right| \rightarrow 0.\] It remains to choose a subsequence $\{E_{n_i}\}$, $i \in \mathbb{N}$ such that \[\max\limits_{|g| \le k}  \left|\frac{f(gE_{n_i})}{f(E_{n_i})}-1\right| < \beta(k).\] 
	
	On the other hand, if any word $g \in S^*$ of length not exceeding $n$ satisfies $\left|\frac{f(gE_n)}{f(E_n)}-1\right| < \beta(n)$, then for a word $g$ of length $k$ we have that $\left|\frac{f(gE_n)}{f(E_n)}-1\right| < \frac{1}{n}$ as long as $n \ge k$. This immediately applies convergence to 0.
\end{proof}

As an example, we note that all bounded positive harmonic functions have an $E_n$-approximation.

\begin{lemma}
	Let $h$ be a bounded positive harmonic function with respect to the simple random walk on a locally finite connected graph $X$. Then there is a sequence $\{x_n\}$ in $X$ such that for any word $g \in S$ the sequence $\frac{h(gx_n)}{h(x_n)}$ converges to 1.
\end{lemma}

\begin{proof}
	We will give a proof for the case of a lazy simple random walk: the generating set $S$ is symmetric, contains the identity element and has $d$ elements in total. We also assume $\mu(g) = \frac{1}{d}$ for any $g \in S$. The more general case can be proven in a similar but slightly more technical way.
	
	Put $\beta(n) = \frac{1}{n}$. Without loss of generality, assume $\sup h = 1$. If the function is constant, the statement is obvious. Otherwise, take a point $y_n$ such that $f(y_n)>1-\varepsilon$, where $\varepsilon = \frac{1}{n d^n}$.
	
	If $g$ is a one-letter word, by harmonicity we get $h(gy_n)>1-d\varepsilon$, because $h(gy_n)>1-\varepsilon$ is the average of $d$ values not exceeding $1$. In analogy, if $g$ is a two-letter word, we get $h(gy_n)>1-d\cdot d\varepsilon=1-d^2\varepsilon$. Continuing by induction, it can be seen that for any word $g$ of length $n$ or less $1 \ge h(gy_n) > 1 - d^n \varepsilon = 1-\frac{1}{n} $, which implies that $\left| \frac{h(gy_n)}{h(y_n)} -1 \right| < \frac{1}{n}$. 
\end{proof}

The following "weak" version of the criterion is also true.

\begin{theorem}
	The action of $G$ on $X$ is amenable if and only if for all positive superharmonic functions $f$ with respect to $\mu$ there exists a sequence $\{E_n\}$ in $X$ such that for all $x \in G$, $\frac{Pf(E_n)}{f(E_n)} \rightarrow 1$ as $n \rightarrow \infty$.
\end{theorem}

\begin{proof}
	If the action is amenable, then for any suitable function $f$ there is a sequence $\{E_n\}$ satisfying the stronger condition from Theorem \ref{north-strong}, which implies in particular the statement of this theorem.
	
	The opposite direction is proven using Lemma \ref{one} and the counterexample from the proof of Theorem \ref{north-strong}.
\end{proof}

We first show the proof of existence of $E_n$-approximations for generic min-functions using several lemmas which are going to be used throughout the paper. After that, we prove similar results for their finite and then countable sums, as well as for linear combinations of their images under the Markov operator. We finish by defining \textit{generalized min-functions} and constructing a sequence of $E_n$ for this new class.

The obtained results do not give a conclusive answer to the question of the (non)-amenability of the Thompson group $F$ since we do not exhaust the class of \textit{all} superharmonic functions, but they might indicate a new direction to research. 

In particular, we do not know how to construct the approximations for infinite convergent sums $\sum\limits_{i=0}^\infty P^{n_i} f_i$. Our proof for finite sums (Lemma \ref{sigen} and Theorem \ref{fourth}) uses the exact expression for iterations of the Markov operator and cannot be easily extended to the infinite case. Finding an approximation for this class could be an interesting question since some of its functions seem to be more closely related to Green's functions.

In section \ref{almost-f2}, we also note that this test can be used to show that some similarly defined actions on other graphs are not extensively amenable.

\textbf{Acknowledgement.} The author is thankful to his advisor Kate Juschenko for helpful research advice, productive discussions and reviewing the presented results.

\section{Overview of Thompson's group}
 Thompson's group $F$ was introduced by Richard Thompson in unpublished handwritten notes in 1965 together with two other related (non-amenable) groups which are not covered by this paper. The definition of Thompson's group $F$, its basic properties with proofs and some more elementary results can be found in various papers and publications: \cite{burillo}, \cite{cfp}, \cite{yeow}.
 
 The question regarding the amenability of $F$ has been open for a long time. If it is amenable, it would be an example of a finitely presented group which is amenable but not elementary amenable (the latter is proven, for instance, in \cite{bcw}). If it is not amenable, it would be an example of a finitely presented non-amenable group not containing the free group $\mathbb{F}_2$ (see \cite{burillo}).
 
 In 2009, a paper published by Azer Akhmedov (\cite{akhmedov}) claimed to prove the non-amenability of $F$, while another paper published the same year by Evgeni Shavgulidze (\cite{shavg}) claimed to prove its amenability. Both proofs turned out to be erroneous.

\section{Extensive amenability as a tool}\label{ext-am}

Extensively amenable actions have several equivalent definitions, see \cite{JMMS}.

Let a discrete group $G$ act on a set $X$ and let $\pfx$ denote the class of all finite subsets of $X$.

\begin{definition}
	Let $G$ act on a set $X$. The action can be canonically extended to $\pfx$ by the rule $g(E)=gE=\{g(x)|x \in E \}$. The original action of $G$ on $X$ is called \emph{extensively amenable} if there is a $G$-invariant mean on $\pfx$ giving the full weight to the collection of subsets containing some (=any) given element of $X$.
\end{definition}

\begin{definition}
	The semidirect product $\pfx \rtimes G$ equipped with multiplication given by $(E,g)(F,h) = (E \Delta gF, gh)$ is called the \emph{Lamplighter group} of the action of $G$ on $X$.
\end{definition}

\begin{definition}\label{ext-lamp}
	An action of $G$ on $X$ is called \emph{extensively amenable} if the action of the Lamplighter group $\pfx \rtimes G$ on $\pfx$, defined by $(E,g)(F) = E \Delta gF$, is amenable.
\end{definition}

Extensively amenable actions are amenable unless $X = \varnothing$, and every action of an amenable group is extensively amenable. Neither converse is true: in Chapter \ref{almost-f2} we show an example of an amenable action which is non extensively amenable, while some examples of extensively amenable actions by non-amenable groups can be found in \cite{JdlS} and \cite{J-book}.

The following lemma states that the set of finite subsets can be replaced with a broader set of finitely supported functions:

\begin{theorem}\label{ext-am-gen}
	Let $G$ act on a set $X$. Then the action is extensively amenable if and only if for some (=any) non-trivial amenable group $A$ the canonical action of $\bigoplus\limits_X A \rtimes G$ on $\bigoplus\limits_X A$ is amenable.
\end{theorem}

This lemma is a direct corollary of Theorem 1.3 from \cite{JMMS}. For the \textit{only if} direction, a simpler proof is given in Lemma 5.5 of \cite{J-book}.

As mentioned before, it is a well-known open question to decide whether Thompson's group $F$ is amenable. In this chapter, we intend to paraphrase the question in terms of extensive amenability of a certain action of $F$. In its turn, extensive amenability can be reformulated in terms of classic amenability of the action produced by the corresponding Lamplighter group by Definition \ref{ext-lamp}. The core idea is to apply Northshield's test to the action.

The following theorem was proved in \cite{JMMS}:

\begin{theorem}\label{embeddings}
	Let $G \curvearrowright X$ be an extensively amenable action and $A$ a non-trivial amenable group. If there exists an embedding $G \hookrightarrow \lampprod $ of the form $g \rightarrow (c(g), g)$ with the property that the kernel $\{g \in G : c(g)=id \}$ is an amenable subgroup of $G$, then $G$ is amenable.
\end{theorem}

\begin{corollary}
	Thompson's group $F$ is amenable if and only if its action on $X = \mathbb{Z}[\frac{1}{2}] \cap (0,1)$ is extensively amenable.
\end{corollary}

\begin{proof}
	The \textit{only if} part is obvious because any action of an amenable group is extensively amenable. 	
	
	For the \textit{if} part, the idea is to construct an embedding $F \hookrightarrow \bigoplus\limits_X A \rtimes F$ of the form $g \rightarrow (c(g),g)$ and apply Theorem \ref{embeddings}. 
	
	Let $A$ be the group of multiplicative integer powers of 2 and let $X$ be the set $\mathbb{Z}[\frac{1}{2}] \cap (0,1)$ of dyadic numbers. Define $c(g) : X \rightarrow A$ by $c(g)(x)=\frac{g_{+}^{'} (x)}{g_{-}^{'} (x)}$. It can be checked that the embedding $g \rightarrow (c(g),g)$ is a well-defined cocycle. The kernel of the cocycle is trivial and hence amenable, which proves the amenability of $F$ under the initial assumptions. 
	
\end{proof}

\section{Schreier graph of Thompson's group}

We consider the left action of $F$ on the set of dyadic rationals $X = \mathbb{Z}[\frac{1}{2}] \cap (0,1)$. The Schreier graph of this action can depend on the generating set. Some example of such graphs can be found in \cite{notstram} and \cite{savchuk}.

The following lemma is proven in several overview papers of the group, for example, \cite{cfp}.

\begin{lemma}\label{f2gen}
	Let $g_0,g_1 \in F$ be defined by:
	
	\[
	g_0(x) = \begin{cases}
	\frac{x}{2}, & 0 \le x \le \frac{1}{2} \\
	x-\frac{1}{4}, & \frac{1}{2} < x \le \frac{3}{4} \\
	2x - 1, & \frac{3}{4} < x \le 1
	\end{cases}
	\]
	
	and 
	
	\[
	g_1(x) = \begin{cases}
	x, & 0 \le x \le \frac{1}{2} \\
	\frac{x}{2}+\frac{1}{4}, & \frac{1}{2} < x \le \frac{3}{4} \\
	x-\frac{1}{8}, & \frac{3}{4} < x \le \frac{7}{8} \\
	2x - 1, & \frac{7}{8} < x \le 1.
	
	\end{cases}
	\]
	
	Then $g_0$ and $g_1$ generate $F$.
\end{lemma}

Let $g_0$ and $g_1$ be the generators of $F$ as defined in Lemma \ref{f2gen}.

Take the Schreier graph $X$ corresponding to generators $a = g_1 g_0^{-1}$ and $b = g_1$, with the right-to-left order of multiplication.

The graph consists of a binary tree (which we call its \textit{skeleton} and denote by $X_{sk}$) and \textit{hairs} attached to its vertices, two to the root (point $p=\frac{5}{8}$) and one to any other vertex. From the transitivity properties of Thompson's group it follows that each dyadic number in $\mathbb{Z}[\frac{1}{2}] \cap (0,1)$ has a unique corresponding vertex in the graph.

\begin{figure}[h]
	\centering
	\includegraphics[width=10cm]{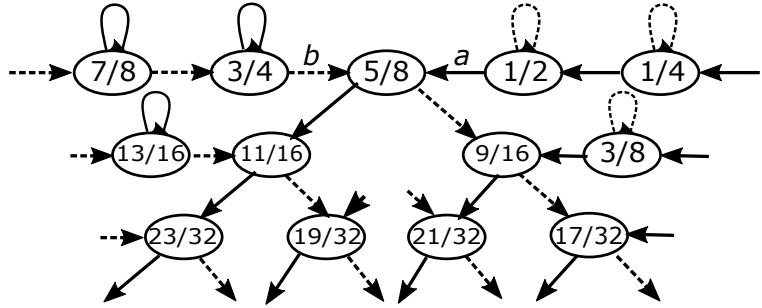}
	\caption{Schreier graph of $F$ acting on $\mathbb{Z}[\frac{1}{2}]$}
\end{figure}

Further on, we will only be interested in the general structure of the graph. The exact dyadic numbers corresponding to individual vertices are of little interest to us.

We also note that the group $\pfx \rtimes F$ is generated by a finite set $\\ \{ (\varnothing, a), (\varnothing, b), (\{p\}, e) \}$. When no confusion occurs, we may abuse the notation and call them $a$, $b$ and $\sigma$ respectively. Roughly speaking, if $E \in \pfx$ is a set, $a$ and $b$ move all its points as the generators $a$ and $b$ of the Thompson's group respectively, and $\sigma$ makes a "switch" at the point $p$ by removing it from the set if it is contained there or adding it to the set otherwise.

\begin{theorem}
	The action of $F$ on $X$ is amenable.
\end{theorem}

\begin{proof}\label{amenable-action}
	We can take arbitrarily long subsequences of hairs as F\o lner sets.
\end{proof}

\section {Northshield's criterion extended to group actions}\label{north}

The amenability criterion below is a generalization of the criterion proved by Sam Northshield and published in \cite{northshield}. The original proof only deals with amenable groups, whereas we extend the statement to all amenable actions of a group $G$ on a graph $X$. The structure of our proof is based on the original.

Here the measure $\mu$ on $G$ is assumed to be generating, symmetric and aperiodic in the sense of Markov chains, the action of $G$ is left and $P$ is the Markov operator with respect to $\mu$.

The following fact is proven in \cite{woess} (Lemma 7.2, p. 81):

\begin{lemma}\label{one}
	The action of $G$ on $X$ is amenable if and only if $\inf\{ \lambda : \exists f > 0: \ Pf \le \lambda f \} = 1$.
\end{lemma}

\begin{theorem}\label{north-strong}
	The action of $G$ on $X$ is amenable if and only if for all positive superharmonic functions $f$ with respect to $\mu$ there exists a sequence $(E_n)$ in $X$ such that for all $x \in G$, $\frac{f(xE_n)}{f(E_n)} \rightarrow 1$ as $n \rightarrow \infty$.
\end{theorem}

\begin{proof}
	Let $G$ act amenably on $X$ and let $f$ be a positive superharmonic function. Define $P^{(n)}$ to be the $n$-fold convolution of $P$ with itself and suppose that $\sup\limits_E \frac{\pn f(E)} {f(E)} \neq 1$. Then there exists $\varepsilon \in (0,1) $ such that $\pn f \le \varepsilon^n f$. Let $g = \sum\limits_{0 \le i \le n-1} \frac{P^{(i)}f}{\varepsilon^i}$. Then $Pg - \varepsilon g = \frac{\pn f(E)}{\varepsilon^{n-1}} - \varepsilon f \le 0 $, which is, by Lemma \ref{one}, a contradiction. Hence $\sup\limits_E \frac{\pn f(E)} {f(E)} = 1$.
	
	For $E \in X$, define $f_E(x)=\sqrt{\frac{f(xE)}{f(E)}}$, where $x \in G$. $f_E^2$ is positive superharmonic, since \[ \sum\limits_g \mu(g) \frac{f(gxE)}{f(E)} = \frac{1}{f(E)} \sum\limits_g  \mu(g) f(g \cdot xE) =  \frac{Pf(xE)}{f(E)} \le \frac{f(xE)}{f(E)}. \] 
	
	Since $t \rightarrow t^{1/2}$ is an increasing concave function, by Jensen's inequality $f_E$ is also positive and superharmonic. Note that
	
	\begin{multline*}
		\pn f_E(e) = \sum\limits_x \mu^{(n)} (x) f_E(x) = \sum\limits_x \mu^{(n)} (x) \sqrt{\frac{f(xE)}{f(E)}} = \\ = \sum\limits_F \pn (E,F) \sqrt{\frac{f(F)}{f(E)}} = \frac{\pn [f^{1/2}](E)}{f(E)^{1/2}} .
	\end{multline*}
	
	Thus, $\sup\limits_E \pn f_E(e) = 1$.
	
	Since, for all $E$, $\pn f_E(e)$ is decreasing as a function of $n$, we can choose $E(k)$ such that, for all $n$, $\pn f_{E(k)}(e) \rightarrow 1$ as $k$ goes to infinity. Then we have: 
	
	$1 = \left( f_{E(k)}(e) \right)^2 \ge \pn [f_{E(k)}]^2(e) \ge [\pn f_{E(k)}(e)]^2 \rightarrow 1$. Thus:
	
	$\sum\limits_x \mu^{(n)} (x) \left(f_{E(k)}(x)-1\right)^2 = \pn [f_{E(k)}]^2(e) - 2 \pn f_{E(k)}(e)+1 \rightarrow 0$, $k \rightarrow \infty$. Since $\mu$ is aperiodic, we have $f_{E(k)} \rightarrow 1$ pointwise.
	
	To prove the opposite direction, assume that for every positive superharmonic function $f$ there is a sequence $(E_n)$ satisfying the condition of the theorem. Then $\frac{Pf(E_n)}{f(E_n)} = \sum\limits_x \mu(x) \frac{f(xE_n)}{f(E_n)} \rightarrow 1$, which implies amenability by Lemma \ref{one}.
	
\end{proof}

In particular, all bounded harmonic functions have this approximation.

\begin{lemma}
	Let $h$ be a bounded positive harmonic function with respect to the simple random walk on a locally finite connected graph $X$. Then there is a sequence $x_n$ in $X$ such that for any word $g$ of generating elements and their inverses the sequence $\frac{h(gx_n)}{h(x_n)}$ converges to 1.
\end{lemma}
	
\begin{proof}
	We will give a proof for the case of a lazy simple random walk: the generating set $S$ is symmetric, contains the identify element and has $d$ elements in total. We also assume each step is a left multiplication by an arbitrary element from the set $S$ chosen with probability $\frac{1}{d}$. The more general case can be proven similarly.
	
	The statement is equivalent to the following: for each $n$ there is an element $x_n \in X$ such that $\left| \frac{h(gx_n)}{h(x_n)} -1 \right| < \frac{1}{n}$ for any word $g$ of length $\le n$. 
		
	Without loss of generality, assume $\sup h = 1$. If the function is constant, the statement is obvious. Otherwise, take a point $y_n$ such that $f(y_n)>1-\varepsilon$, where $\varepsilon = \frac{1}{n d^n}$.
		
	If $g$ is a one-letter word, by harmonicity we get $h(gy_n)>1-d\varepsilon$, where $d=2k+1$ is greater than or equal to the degree of $y_n$. This happens because $h(gy_n)>1-\varepsilon$ is an average of $d$ values not exceeding $1$. Similarly, if $g$ is a two-letter word, we get $h(gy_n)>1-d\cdot d\varepsilon=1-d^2\varepsilon$. Continuing in the same way, we obtain that for any word $g$ of length $n$ or less $1 \ge h(gy_n) > 1 - d^n \varepsilon = 1-\frac{1}{n} $, which implies that $\left| \frac{h(gy_n)}{h(y_n)} -1 \right| < \frac{1}{n}$. 
\end{proof}

Since an example of a suitable non-superharmonic function for a non-amenable group action can be constructed in a universal manner for all actions, the criterion can be reformulated as follows:

\begin{theorem}\label{north-weak}
	The action of $G$ on $X$ is amenable if and only if for all positive superharmonic functions $f$ with respect to $\mu$ there exists a sequence $(E_n)$ in $X$ such that for all $x \in G$, $\frac{Pf(E_n)}{f(E_n)} \rightarrow 1$ as $n \rightarrow \infty$.
\end{theorem}

\begin{proof}
	If the action is amenable, then for any suitable function $f$ there is a sequence $E_n$ satisfying the stronger condition from Theorem \ref{north-strong}, which implies in particular the statement of this theorem.
	
	The opposite direction is proven using Lemma \ref{one} in a way similar to Theorem \ref{north-strong}.
\end{proof}

\section{Strong and weak approximations}
Depending on the context we can choose to use the strong (as in Theorem \ref{north-strong}) or weak (as in Theorem \ref{north-weak}) approximation. If the aim is to prove that a group is non-amenable, the strong criterion works better because the strong approximation has a better chance of catching a "bad" superharmonic function. On the other hand, to conclude that a group is amenable, it is easier to verify that all superharmonic functions have an approximation in a weak sense.

In this paper we will be primarily focused on strong approximations, with making references to weak approximations when necessary.

\begin{theorem}
	Let $f$ be a function of $\pfx$ such that any finite set $E$ satisfies $f(E)=f(E \Delta \{p\})$ (we will call such functions switch-invariant). Then $f$ satisfies the weak approximation condition: there is a sequence $E_n$ such that $\frac{Pf(E_n)}{f(E_n)} \rightarrow 1$.
\end{theorem}

\begin{proof}
	Let $\varphi(x) = f(\{x\})$ be the restriction of $f$ to the class of one-element sets. Since $f$ is superharmonic and switch-invariant, so is $\varphi_0$ with respect to the simple random walk generated by $\{a,b,a^{-1},b^{-1}\}$. The amenability of the action of $F$ on $X$ (Theorem \ref{amenable-action}) implies that there is a sequence $\{x_n\} \subset X$ such that $\frac{P\varphi(x_n)}{\varphi(x_n)} \rightarrow 1$. On the other hand, $P\varphi(x_n) = Pf(\{x_n\})$ and $\varphi(x_n)=f(\{x_n\})$ by definition and properties of $f$. This, in turn, gives an approximation in a weak sense for $f$.
\end{proof}

It is not known whether a strong approximation exists for this class. Further on in the paper, we will find such approximations for a subclass of switch-invariant functions as well as some other functions. 

\section{Min-functions are well-defined and superharmonic}\label{min-well-def}

\begin{definition}
		Let $\varphi$ be a function defined on a set $X$ and achieving its maximum at the root point $p$. Then the function $f : \mathcal{P}_f(X) \rightarrow \mathbb{R} $ defined by 
		\begin{itemize}
			\item $f(E) = \min\limits_{x \in E} \varphi(x)$ 
			\item $f (\varnothing) = \varphi (p)$
		\end{itemize}
		
		is called the min-function of $\varphi$.
\end{definition}

\begin{lemma}
	Let $\varphi$ be a superharmonic function on $X$ achieving its maximum at the root point $p$. Then its min-function $f$ is also superharmonic with respect to any measure given by $\mu (\varnothing, a)=\mu (\varnothing, b)=\mu (\varnothing, a^{-1})=\mu (\varnothing, b^{-1})>0$, $\mu (\{p\}, id)>0$.
\end{lemma}

\begin{proof}
	Define 
	\begin{itemize}
		\item $T_1f(E)=\frac{1}{4} (f(aE)+f(bE)+f(a^{-1}E)+f(b^{-1}E))$
		\item $T_2f(E)=f(E \Delta p)$
		\item $Tf(E)=\alpha T_1f(E)+(1-\alpha) T_2f(E)$, $\alpha \in (0,1)$.
	\end{itemize}

	Let $E$ be a non-empty finite subset of $X$ and let $x$ be the point in $E$ minimizing $\varphi$. Then by definition:
	
	$f(aE) \le \varphi(a.x)$;
	
	$f(bE) \le \varphi(b.x)$;
	
	$f(a^{-1}E) \le \varphi(a^{-1}.x)$;
	
	$f(b^{-1}E) \le \varphi(b^{-1}.x)$;
	
	$T_1f(E)=\frac{1}{4}(f(aE)+f(bE)+f(a^{-1}E)+f(b^{-1}E)) \le \frac{1}{4}(\varphi(a.x)+\varphi(b.x)+\varphi(a^{-1}.x)+ \varphi(b^{-1}.x)) =P\varphi(x) \le \varphi(x) = f(E)$
	
	$T_2f(E) = f(E)$ since by definition $\varphi(p) \ge \varphi(x)$.
	
	Hence, $Tf(E) \le f(E)$ for any non-empty $E$. For the empty set the statement is easily checked by hand.
\end{proof}

\begin{lemma}\label{not-supharm}
	Let $\varphi$ be a function on $X$ such that:
	\begin{itemize}
		\item $\varphi$ is not superharmonic;
		\item $\varphi(x) \le \varphi(p)$ for any $x \in X$.
	\end{itemize}
	
	Then the corresponding min-function $f$ is not superharmonic either.
\end{lemma}

\begin{proof}
	Let $q$ be a point where $\varphi$ is not superharmonic, i.e. $\varphi(q)<P\varphi(q)$. Clearly $q \neq p$. Let $r_i$, $i=\overline{1,4}$, be the neighbors of $q$ (possibly multiple or coinciding with $q$). Then $4\varphi(q)<\varphi(r_1)+\varphi(r_2)+\varphi(r_3)+\varphi(r_4).$
	
	Now assume that $f$ is superharmonic. Rewriting the superharmonicity condition for the set $\{q\}$, we get \[ 5f(\{q\}) \ge f(\{r_1\})+f(\{r_2\})+f(\{r_3\})+f(\{r_4\})+f(\{q,p\}). \] In terms of $\varphi$, this is equivalent to \[ 5\varphi(q) \ge \varphi(r_1)+\varphi(r_2)+\varphi(r_3)+\varphi(r_4)+\varphi(q) \] or \[ 4\varphi(q) \ge \varphi(r_1)+\varphi(r_2)+\varphi(r_3)+\varphi(r_4), \] which contradicts the above.
	
	Therefore, $f$ is not superharmonic at the point $\{q\}$.
\end{proof}

\section{Results for min-functions and their modifications}

\begin{theorem}\label{first}
	Let $X$ be the Schreier graph of Thompson's group $F$ with $p$ being the root vertex. Let $\varphi$ be a positive superharmonic (with respect to the uniform measure) function on $X$ such that $f(q) \le f(p)$ for all points $q \in X$ and $f$ be its min-function.
	
	Then there is a sequence $E_n \in \mathcal{P}_f(X)$, $n \in \mathbb{N}$, such that for any $g \in \pfx \rtimes F$ the sequence $\frac{f(gE_n)}{f(E_n)}$ converges to 1.
\end{theorem}

Before proving the theorem we will need to introduce some additional notation and prove several intermediate lemmas.

Let $z$ be a point on the binary skeleton $X_{sk}$. Denote by $\hair{z}{m}$ the point located on the corresponding hair $m$ points away from $z$.

We can see $g$ as a word composed of letters in $S=\{a,a^{-1},b,b^{-1}, \sigma\}$ corresponding to the generators of $\pfx \rtimes F$ and their inverses:

\begin{itemize}
	\item $aE = \{a.x \ | \ x \in E \}$;
	\item $a^{-1}E = \{a^{-1}.x \ | \ x \in E \}$;
	\item $bE = \{b.x \ | \ x \in E \}$;
	\item $b^{-1}E = \{b^{-1}.x \ | \ x \in E \}$;
	\item $\sigma E = E \Delta \{p\}$.
\end{itemize}

Denote by $S^*$ the set of all finite words in $S$. If $g$ has several representations in $S^*$, we pick the shortest one. If there are several minimal-length representations, we can pick an arbitrary one. 

\begin{proposition}\label{concave}
	$\varphi(\hair{z}{m})$ is always a concave function with respect to $m$, i.e. $\linebreak \varphi(\hair{z}{m+1}) - \varphi(\hair{z}{m})$ is non-increasing.
\end{proposition}
\begin{proof}
	Apply the superharmonicity of $\varphi$ to $\hair{z}{m}$, where $m \ge 1$: 
	\[ \varphi(\hair{z}{m}) \ge \frac{1}{4} (\varphi(\hair{z}{m+1})+\varphi(\hair{z}{m-1})+2\varphi(\hair{z}{m})).  \] Indeed, from $\hair{z}{m}$ by the distribution law we can get to $\hair{z}{m-1}$, $\hair{z}{m+1}$ and $\hair{z}{m}$ itself with probabilities $\frac{1}{4}$, $\frac{1}{4}$ and $\frac{1}{2}$ respectively. Multiplying the expression by 4 and simplifying gives \[ 2\varphi(\hair{z}{m}) \ge \varphi(\hair{z}{m-1}) + \varphi(\hair{z}{m+1}), \] which is equivalent to \[ \varphi(\hair{z}{m}) - \varphi(\hair{z}{m-1}) \ge \varphi(\hair{z}{m+1}) - \varphi(\hair{z}{m}). \]
\end{proof}

\begin{proposition}\label{nondecr}
	$\varphi$ is non-decreasing on hairs.
\end{proposition}
\begin{proof}
	Let $q$ be a point on the skeleton. Assume that there is $m \ge 0$ such that $\varphi(\hair{z}{m}) > \varphi(\hair{z}{m+1})$. By Lemma \ref{concave} $\varphi(\hair{z}{m})$ is concave, therefore \[\varphi(\hair{z}{m}) - \varphi(\hair{z}{m+k}) \ge k\left( \varphi(\hair{z}{m}) - \varphi(\hair{z}{m+1})\right) \] and \[\varphi(\hair{z}{m+k}) \le \varphi(\hair{z}{m}) - k\left( \varphi(\hair{z}{m}) - \varphi(\hair{z}{m+1})\right). \] This implies that for a sufficiently large $k$ the value $\varphi(\hair{z}{m+k})$ is negative, which contradicts the positivity assumption.
\end{proof}

\begin{proposition}\label{est}
	For all $m \ge 0$, $\varphi(z^{(+m)}) \le (3m+1)\varphi(z) $.
\end{proposition}
\begin{proof}
	By superharmonicity, $z$ has four neighbors $z_1$, $z_2$, $z_3$ and $z^{(+1)}$, so $\varphi(z) = \frac{1}{4}(\varphi(z_1)+\varphi(z_2)+\varphi(z_3)+\varphi(z^{(+1)})) \ge \frac{1}{4}\varphi(\hair{z}{1})$ and $\varphi(\hair{z}{1}) \le 4\varphi(z)$. The rest follows from Proposition \ref{concave} and Jensen's inequality: \[ \varphi(\hair{z}{m}) \le m\varphi(\hair{z}{1})-(m-1)\varphi(z) \le 4m\varphi(z)-(m-1)\varphi(z)=(3m+1)\varphi(z). \]
\end{proof}

\begin{proposition}\label{minimum}
	Let $X_n$ be the subtree of the binary skeleton consisting of all points belonging to the upper $n+1$ levels of the tree without hairs (i.e., it contains $2^{n+1}-1$ points). Denote by $r_n$ the minimum of $\varphi$ on $X_n$. Then there is a point $q_n$ on the $(n+1)$-st level such that $\varphi(q_n)=r_n$.
\end{proposition}

\begin{proof}
	
	Assume that $r_n$ is achieved at some point $z$ above the $(n+1)$-st level. Then, by superharmonicity and Proposition \ref{nondecr}, all neighbors $z_0$ of $z$ satisfy $\varphi(z_0) =r_n$. Now we can take the neighbor of $z$ which is one level down and apply the same argument to it. Repeating the same procedure until we reach the $(n+1)$-st level, we obtain the desired conclusion.
	
\end{proof}

\begin{proof} \textit{(of Theorem \ref{first})}
	
	The statement is obvious if $\varphi \equiv C$ (and hence $f \equiv C$) for some $C \in \mathbb{R}$, so further on we assume that $\varphi$ is non-constant.

	Since $\varphi$ is positive and $r_n$ is non-increasing, there exists a number $r \ge 0$ such that $r_n \downarrow r$, and, by Lemma \ref{nondecr}, $r$ is the infimum of $\varphi$ over $X$. Replacing $\varphi (x)$ with $\varphi(x) - r$ preserves superharmonicity and makes the statement to prove even stronger: if $\frac{f(gE_n)-r}{f(E_n)-r}$ converges to 1, then so does $\frac{f(gE_n)}{f(E_n)}$. Since by assumption $\varphi$ is non-constant, it remains strictly positive everywhere together with $f$ because $\varphi$ cannot have minimums. So it can be assumed without loss of generality that $r_n \downarrow 0$.
	
	The idea of the proof is to construct a sequence of one-element sets $E_n=\{y_n\}$ such that any word $g$ of length not exceeding $n$ would satisfy $\left|\frac{f(gE_n)}{f(E_n)}-1\right| < \frac{1}{n}$. By Proposition \ref{north-equiv}, the existence of such a sequence is equivalent to the approximation in a strong sense. 
	
	Let $n \ge 4$ be a positive integer. Since by our assumption $r_n \downarrow 0$, there is a number $N = N(n) \in \mathbb{N}$ such that $r_N = \varphi(q_N) < \frac{r_n}{4n^2}$. Denote by $a_m=\varphi(\hair{(q_N)}{m})$, $m \ge 0$, the sequence of values of $\varphi$ on the hair attached to $q_n$. Put $y_n=\hair{q_N}{n^2}$ and $E_n = \{y_n\}$. By definition, $f(E_n)=\varphi(y_n)=a_{n^2}$.
	
	If the word representing $g$ contains only moves by $a$, $b$ and their inverses, then $gE_n=\{a_{n^2+i}\}$, where $-n \le i \le n$. To check the condition $\left|\frac{f(gE_n)}{f(E_n)}-1\right| < \frac{1}{n}$ for this case, it is sufficient to verify that $\left|\frac{a_{n^2-n}}{a_{n^2}}-1\right|=\frac{a_{n^2}-a_{n^2-n}}{a_{n^2}} < \frac{1}{n}$ and $\left|\frac{a_{n^2+n}}{a_{n^2}}-1\right|=\frac{a_{n^2+n}-a_{n^2}}{a_{n^2}} < \frac{1}{n}$. By concavity, $a_{n^2+n}-a_{n^2} \le a_{n^2}-{a_{n^2-n}}$, so it is enough to verify the former inequality. Using concavity again, it can be seen that $a_{n^2}-{a_{n^2-n}} \le \frac{1}{n}(a_{n^2}-{a_0}) < \frac{a_{n^2}}{n}$, hence, $\frac{a_{n^2}-a_{n^2-n}}{a_{n^2}} < \frac{1}{n}$.
	
	If $g$ contains $\sigma$ in its representation at least once, it means that in the process of applying the word's instruction to $E_n$ we might add some new points and, possibly, move them along tree edges. This can only affect the value of $f(gE_n)$ if for some subword $g_1 \in \{a,b,a^{-1}, b^{-1}\}^*$ with $|g_1| \le |g| \le n$ we have that $\varphi(g_1.p)<a_{n^2+n}$. But by Propositions \ref{nondecr} and \ref{minimum} any value we can achieve this way is greater than or equal to $r_n$. Applying Lemma \ref{est}, we get \[ a_{n^2+n} = \varphi (\hair{q_N}{n^2+n}) \le (3n^2+3n+1)\varphi (q_N) \le (3n^2+3n+1)\frac{r_n}{4n^2}< r_n, \] which completes the proof.
	
\end{proof}

Now we show that the approximation also exists for finite sums of min-functions, however, in this case one-element subsets might not be sufficient. 

\begin{lemma}\label{mins}
	Let $a_1, \ldots, a_m, b_1, \ldots, b_m, \alpha, \beta$ be positive real numbers such that $\frac{a_i}{b_i} \in (\alpha, \beta)$ for all $\oneik$. Then $\frac{\min(a_1, \ldots, a_n)}{\min(b_1, \ldots, b_n)} \in (\alpha, \beta)$.
\end{lemma}

\begin{proof}
	Put $a_i=\min(a_1, \ldots, a_n)$ and $b_j=\min(b_1, \ldots, b_n)$. Then:
	
	$\frac{\min(a_1, \ldots, a_n)}{\min(b_1, \ldots, b_n)}=\frac{a_i}{b_j} \in [\frac{a_i}{b_i}, \frac{a_j}{b_j}] \subset (\alpha, \beta)$.
\end{proof}

\begin{theorem}\label{second}
	Let $f_i$, $i=\overline{1,k}$ be a set of superharmonic functions on $\mathcal{P}_f(X)$ obtained as in Theorem \ref{first}. Put $f=\sum\limits_{i=1}^k \lambda_i f_i$, where $\lambda_i$ are positive real numbers. Then $f$ retains the same property: there is a sequence $E_n \in \mathcal{P}_f(X)$, $n \in \mathbb{N}$, such that for any word $g \in \pfxf$ the sequence $\frac{f(gE_n)}{f(E_n)}$ converges to 1.
\end{theorem}

\begin{proof}
	
	Without loss of generality it can be assumed that $\lambda_i=1$ since the conditions of Theorem \ref{first} are invariant under scaling by a positive number.
	
	Let $\varphi_i : X \rightarrow \mathbb{R}$ be the functions producing $f_i$. As in the previous proof, we can assume $\inf \varphi_i = 0$ for all $i$.
	
	Denote by $\ball$ the set of all points in $X$ at a distance at most $n$ from $p$ and put $\varepsilon= \min\limits_{\oneik} f_i(\ball) = \min\limits_{\oneik, \ x \in \ball} \varphi_i(x)$.
	
	Now for each $\oneik$ let $x_i$ be an arbitrary point on $X_{sk}$ satisfying $\varphi_i(x_i)<\frac{\varepsilon}{4n^2}$. We claim the set $E_n = \{x_1^{(+n^2)}, \ldots, x_k^{(+n^2)}\}$ satisfies $\left|\frac{f(gE_n)}{f(E_n)}-1\right| < \frac{1}{n}$ for words $g$ of length $\le n$.
	
	As in the previous proof, assume first that the instruction given by $g$ contains only moves by $a$, $b$ and their inverses. Then $gE_n = \{x_1^{(+n^2+s_1)}, \ldots, x_k^{(+n^2+s_k)}\}$, where $-n \le s_i \le n$ (in fact, there are only two possible values of $s_i$ depending on whether $x_i$ is a left or right child on the binary tree). For convenience, denote $y_{i}=x_i^{(+n^2)}$ and $z_{i}=x_i^{(+n^2+s_i)}$.

	By definition, $f_i(E_n) = \min\limits_{1 \le j \le n} \varphi_i(y_j)$ and \[ f_i(gE_n)=\min\limits_{1 \le j \le n} \varphi_i(z_j). \] Using the concavity analogously to the proof of Theorem \ref{first}, we have that
	
	\[ \varphi_i(z_j) \in\left( (1-\frac{1}{n})\varphi_i(y_j), (1+\frac{1}{n})\varphi_i(y_j)  \right). \]  
	
	By Lemma \ref{mins} we conclude that \[ f_i(gE_n) \in \left( (1-\frac{1}{n})f_i(E_n), (1+\frac{1}{n})f_i(E_n)  \right). \] Taking the sum over all $i$, we get \[ f(gE_n) \in \left( (1-\frac{1}{n})f(E_n), (1+\frac{1}{n})f(E_n)  \right), \] which is equivalent to $\left|\frac{f(gE_n)}{f(E_n)}-1\right| < \frac{1}{n}$.
	
	Again, if $g$ contains at least one switch, it means that in the process of applying the word's instruction to $E_n$ we might add some new points and, possibly, move them along tree edges. In this case, $gE_n= F_1 \cup F_2 $, where 
	$ F_1 = \{x_1^{(+n^2+s_1)}, \ldots, x_k^{(+n^2+s_k)}\}$ as defined previously and $F_2 \subset \mathcal{B}_n$.
	
	For all $n \ge 4$ we can estimate that
	\begin{multline*}
		f_i(F_1) = \min\limits_{1 \le j \le n} \varphi_i (x_j^{(+n^2+s_j)}) \le \min\limits_{1 \le j \le n} (3n^2+3n+1)\varphi_i(x_j) \le \\ \le (3n^2+3n+1)\varphi_i(x_i) \le \varepsilon \le f_i(\ball) \le f_i(F_2),  
	\end{multline*}
	
	where the last inequality follows from the definition of $f_i$ as the minimum over a set. Finally, we conclude that \[f(gE_n) = \sum\limits_{i=1}^n f_i(gE_n) = \sum\limits_{i=1}^n \min(f_i(F_1), f_i(F_2)) = \sum\limits_{i=1}^n f_i(F_1) = f(F_1).\] From the proof of the first case we know that $\left|\frac{f(F_1)}{f(E_n)}-1\right| < \frac{1}{n}$, which completes the proof.
	
\end{proof}

Another generalization can be made by using the fact that the Markov operator preserves superharmonicity.

\begin{theorem} \label{third}
	Let  $f : \mathcal{P}_f(X) \rightarrow \mathbb{R} $ be a superharmonic function with the strong approximation property: there is a sequence $E_n \in \mathcal{P}_f(X)$, $n \in \mathbb{N}$, such that for any word $g \in \pfxf $ the sequence $\frac{f(gE_n)}{f(E_n)}$ converges to 1. Then $Pf$, where $P$ is the corresponding Markov operator, satisfies the same property for the same sequence $E_n$.
\end{theorem}

\begin{proof}
	Note that $Pf(E)=\frac{1}{5}(f(aE)+f(bE)+f(a^{-1}E)+f(b^{-1}E)+f(E \Delta \{p\}))$. 
	
	Let $g$ be an arbitrary word. Then \[ \frac{Pf(gE_n)}{Pf(E_n)}=\frac{f(agE_n)+f(bgE_n)+f(a^{-1}gE_n)+f(b^{-1}gE_n)+f(gE_n \Delta \{p\})}{f(aE_n)+f(bE_n)+f(a^{-1}E_n)+f(b^{-1}E_n)+f(E_n \Delta \{p\})}. \] From the properties of $E_n$, we have that $\frac{f(agE_n)}{f(aE_n)} = \frac{f(agE_n)/f(E_n)}{f(aE_n)/f(E_n)} \rightarrow 1, \ n \rightarrow \infty$, or $f(agE_n) \sim f(aE_n), \ n \rightarrow \infty$. Summing up the equivalences for all five generators, we obtain the desired conclusion.
	
\end{proof}

Now we can move on to a more general observation involving linear combinations and iterations of the Markov operator.

\begin{lemma} \label{sigen}
	Let $f_i$, $\oneik$, be the min-functions obtained from $\varphi_i$ (not necessarily distinct. Let $s_i$ and $\lambda_i>0$, $\oneik$, be fixed elements of $G$
	and real numbers respectively. Then there is a sequence $\{E_n\}$ such that for any word $g \in \pfxf$:  \[\frac{\sum\limits_{1 \le i \le k} \lambda_i f_i (s_igE_n) }{ \sum\limits_{1 \le i \le k} \lambda_i f_i (s_i E_n)} \rightarrow 1, \ n \rightarrow \infty.\]
\end{lemma}

\begin{proof}
	First we show that for each $n \ge 4$ there is a finite set $E_n$ such that for each $\oneik$ and for any word $g$ of length $\le n$ the following two estimates hold:
	
	\begin{itemize}
		\item $\left| \frac{f_i(s_igE_n)}{f(E_n)}-1 \right| < \frac{1}{n}$;
		\item $\left| \frac{f_i(s_iE_n)}{f(E_n)}-1\right| < \frac{1}{n}$.
	\end{itemize}
	
	This holds because according to the proof of Theorem \ref{second}, there exists a finite set $E$ such that for all elements $g$ such that $|g|\le m = n+\max\limits_{\oneik} |s_i|$ and for all $\oneik$ we have that $|\frac{f_i(gE)}{f_i(E)} -1| < \frac{1}{m} < \frac{1}{n}$. Since the lengths of $s_ig$ and $s_i$ do not exceed $m$, both inequalities are true for $E_n:=E$. Thus, for any $g$ and all $\oneik$:
	
	\[ \frac{f_i(s_igE_n)}{f_i(s_iE_n)} = \frac{f_i(s_igE_n)/f_i(E_n)}{f_i(s_iE_n)/f_i(E_n)} \rightarrow \frac{1}{1} = 1. \]
	
	Fix $\varepsilon > 0$. From the convergence above, there is a number $N \in \mathbb{N}$ such that for all $n \ge N$: 	\[ f_i(s_igE_n) \in \left((1-\varepsilon){f_i(s_iE_n)} , (1+\varepsilon){f_i(s_iE_n)}\right). \] Summing up the inequalities above with corresponding weights, we get \[ \sum\limits_{1 \le i \le k} \lambda_i f_i (s_igE_n) \in \left((1-\varepsilon)\sum\limits_{1 \le i \le k} \lambda_i f_i (s_i E_n) , (1+\varepsilon)\sum\limits_{1 \le i \le k} \lambda_i f_i (s_i E_n)\right), \] which implies \[ \frac{\sum\limits_{1 \le i \le k} \lambda_i f_i (s_igE_n)}{\sum\limits_{1 \le i \le k}\lambda_i f_i (s_i E_n)} \in (1-\varepsilon, 1+\varepsilon). \] The latter is equivalent to convergence to 1.
	
\end{proof}

\begin{theorem}\label{fourth}
	Let $f_i$, $\oneik$, be min-functions and let $P$ be the Markov operator as in Theorem \ref{third}. Then for any integers $n_i \ge 0$, $P^{n_1}f_1 + \ldots + P^{n_k}f_k$ has an approximation in a strong sense.
\end{theorem}

\begin{proof}
	Note that $P^{(n_1)}f_1(E_n) = \sum\limits_s \mu^{(n)}(s) f_1(sE_n)$, where $\mu^{(n)}$ is the $n$-th convolution of the uniform measure and the summation is taken over its support. Now it can be seen that
	
	\[ \frac{\sum\limits_i P^{n_i}f_i(gE_n)} {\sum\limits_i {P^{n_i}f_i(E_n)}} = \frac{\sum\limits_{i, |s_i|\le n_i} \mu^{(n_i)} f_i(s_igE_n)}{\sum\limits_{i, |s_i|\le n_i} \mu^{(n_i)} f_i(s_iE_n)}, \] which converges to 1 by Lemma \ref{sigen}.
\end{proof}

The same result is true for countable sums of min-functions whenever they are well-defined.

\begin{theorem}
	If $\varphi_i, \ i \ge 1$, are functions satisfying the condition above such that $\sum\limits_{i=1}^\infty \varphi_n(p) < \infty$, and $f_i$ are the corresponding min-functions, then there is a sequence $\{E_n\}$ for $f = \sum\limits_{i=1}^\infty f_i$ satisfying $\frac{f(gE_n)}{f(E_n)} \rightarrow 1, \ n \rightarrow \infty,$ for all $g \in \pfxf$.
\end{theorem}

\begin{proof}
	As in the proof of Theorem \ref{second}, denote by $\ball$ the set of all points in $X$ at a distance at most $n$ from $p$. Put $\varepsilon_i = \min\limits_{x \in \ball} f_i(x)$.
	
	Put $t_i = \varphi_i(p)$. Without loss of generality, we can assume $\sum\limits_{i=1}^\infty t_n = \sum\limits_{i=1}^\infty \varphi_n(p) = 1$.
	
	As before, we are trying to construct a sequence of sets $E_n$ such that for all words $g$ of length not exceeding $n$, we have $\left|\frac{f(gE_n)}{f(E_n)}-1\right|<\frac{1}{n}$. 
	
	Fix $n \in \mathbb{N}$. First, we take a point $z_1$ in the skeleton of the tree such that $\delta := \varphi_1(z_1) < \frac{\varepsilon_1}{16n^2}$ and define $y_1=\hair{z_1}{4n^2}$. We can verify that if $|g| \le n$, then $\frac{f_1(g\{y_1\})}{f_1(\{y_1\})} \in (1-\inv{2n}, 1+\inv{2n})$, because if $f(g\{y_1\})$ has any points outside the hair, they are in $\ball$ and hence do not influence the value of $f$ as \[ f(\hair{z_1}{4n^2+n}) \le (12n^2+3n+1) f(z_1) < (12n^2+3n+1)\frac{\varepsilon_1}{16n^2} \le \varepsilon_1. \] Outside $\ball$, the only point $g(\{y_n\})$ can contain is $\hair{z_1}{4n^2+s}, \ |s| \le n$, for which the proof is analogous to that in Theorem \ref{first}. 
	
	Let $N$ be an integer index such that $\sumr{i}{N+1}{\infty} t_i < \frac{\delta}{3n}$. For each $2 \le i \le N$, we can find a point $y_i$ such that $\frac{f_i(g\{y_i\}) } {f_i(\{y_i\})} \in (1-\inv{2n}, 1+\inv{2n})$. 
	
	In order to do it, we choose $z_i$ on the skeleton such that $f_i(z_i) < \frac{\varepsilon_i}{16n^2}$ and then define $y_i=\hair{z_i}{4n^2}$. The proof that $\frac{f_i(g\{z_i\})}{f_i(\{z_i\})} \in (1-\inv{2n}, 1+\inv{2n})$ is equivalent to the case $n=1$.
	
	Finally, define $E_n = \{y_1, y_2, \ldots, y_N\}$. From the construction it follows that $\frac{f_i(gE_n)}{f_i(E_n)} \in  (1-\inv{2n}, 1+\inv{2n})$ for all $1 \le i \le n$. This is true because $gE_n$ consists of two subsets: a subset of $\ball$ and a set of the form $\{\hair{z_1}{4n^2+s_1}, \ldots, \hair{z_N}{4n^2+s_N}\}$. The subset of $\ball$ does not influence the values of $f_i$, and hence $f$ (by construction). For the latter set, we use the argument from Theorem \ref{second}. This statement also implies $\frac{\sumr{i}{1}{N}f_i(gE_n)}{\sumr{i}{1}{N}f_i(E_n)} \in (1-\inv{2n}, 1+\inv{2n})$.
	
	It remains to notice that $\suml_{i=N+1}^\infty f_i(gE_n) < \frac{\delta}{3n}$ and $\suml_{i=N+1}^\infty f_i(E_n) < \frac{\delta}{3n}$. Define $S_1 = \sumr{i}{1}{N} f_i(E_n) = \delta + r$, where $r>0$. Then $f(E_n) \in [\delta+r, \delta(1+\inv{3n})+r]$. 
	
	Now we can see that 
	\begin{multline*}
		f(gE_n) = \sumr{i}{1}{N}f_i(gE_n) + \sumr{i}{N+1}{\infty}f_i(gE_n) \le (1+\inv{2n})(\delta+r)+\frac{\delta}{3n} \le \\ \le (1+\inv{n})(\delta+r) \le (1+\inv{n})f(E_n). 
	\end{multline*}
	
	The central inequality is equivalent to $\frac{\delta}{3n} \le \frac{\delta+r}{2n}$, which is natually true.
	
	Similarly, 
	\begin{multline*} f(gE_n) \ge \sumr{i}{1}{N}f_i(gE_n) \ge (1-\inv{2n})(\delta+r) \ge (1-\inv{n})(\delta(1+\inv{3n})+r) \ge (1-\inv{n})f(E_n), 
	\end{multline*} 
	where the central inequality comes from $\frac{\delta+r}{2n} > \frac{\delta}{2n} > \frac{\delta}{3n}(1-\inv{n})$. The proof is complete.
\end{proof}

\subsection{$E_n$ can have unbounded size}

It is also the case that for some superharmonic functions $f$ the minimal size of $E_n$ as a finite set is not necessarily bounded as a function of $n$.

The idea of one possible counterexample is to construct an infinite sequence of superharmonic functions $\varphi_n : X \rightarrow \mathbb{R}$, each giving rise to a min-function $f_n$ on finite subsets of $X$. As proved before, if the sum $f = \sum\limits_{n=0}^\infty f_i$ is defined, it is also superharmonic and has an $\{E_n\}$-approximation in the strong sense.

Let us introduce some extra notation.

Denote subtrees of $X$ by $T_0, \ T_1, \ldots, T_n, \ldots$ as follows:

$T_0$ includes the right subtree of $X_{sk}$ with all attached hairs. $T_1$ is the right subtree of the underlying binary tree of $X_sk \backslash T_0$, also with all corresponding hairs attached. Continuing in the same fashion, $T_n$ is the right subtree of the underlying binary tree of $X_sk \backslash T_{n-1}$ with hairs. 

\begin{figure}[h]
	\centering
	\includegraphics[width=8cm]{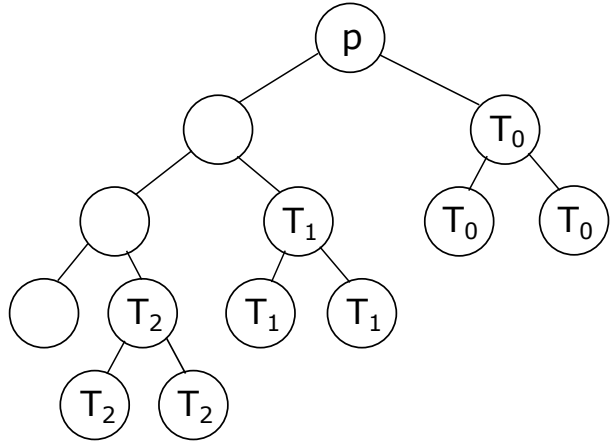}
	\caption{Structure of subtrees $T_i$ (with no hairs)}
	\label{fig:tn}
\end{figure}

Define $\varphi_n$ on $X$ as follows:

$\varphi_n (x) =
\begin{cases}
\frac{1}{2^n}, \ x \notin T_n \\
\frac{1}{2^d}, \ x \in T_n \cap X_{sk}, \  d=d(x,p) \\
\varphi_n(z), \ x=z^{(+m)}, \ m \in \mathbb{N}
\end{cases}
$

Notice that $\varphi_n$ is always constant on hairs. The graphs of $\varphi_n$ on the binary tree for small values of $n$ are given in Fig. \ref{fig:values}.

\begin{figure}[h]\label{fig:values}
	\centering
	\includegraphics[width=15cm]{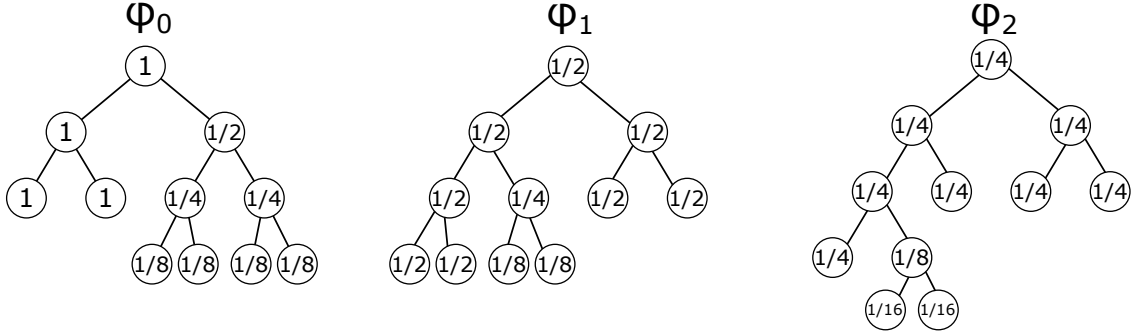}
	\caption{Values of $\varphi_0$, $\varphi_1$, $\varphi_2$ on $X_{sk}$ around $p$}
\end{figure}

For any $n$ and any finite set $E \subset X$, $f_n(E)<\frac{1}{2^n}$, hence $f(E)$ is well-defined, bounded ($0 < f(E) < 2$) and superharmonic.

The amenability criterion from \ref{north-strong} can be reformulated as follows: there is a sequence $\{E_n\}$ such that for any word of length $\le n$ representing an element $g$ we have that $|\frac{f(gE_n)}{f(E_n)} - 1| < \frac{1}{2^n}$. (In previous sections we took the inequality $|\frac{f(gE_n)}{f(E_n)} - 1| < \frac{1}{n}$ instead, but $\frac{1}{2^n}$ is more suited for the purposes of this construction).

\begin{proposition}
	Let $E_n$ be a sequence in $\mathcal{P}_f(X)$ such that for any word $g \in S^*$ of length $\le n$ we have that $|\frac{f(gE_n)}{f(E_n)} - 1| < \frac{1}{2^n}$. Then $|E_n| \rightarrow \infty$, $n \rightarrow \infty$.
\end{proposition}

\begin{proof}
	
	Fix $n \ge 2$. Let us assume that for some $i$, $0 \le i \le n-2$, $E_n$ contains no points in $T_i$. Then it is possible to find a word $g$, $|g| \le i+2$, such that:
	
	$f_i(gE_n) \le f_i(E_n)-\frac{1}{2^{i+1}}$
	
	$f_k(gE_n) \le f_k(E_n)$, $k \neq i$.
	
	Let us construct this word $g$ for a given $i$. Define the \textit{golden path} as the set of points \{$p$, $a.p$, $a^2.p$, \ldots, $a^i.p$, $ba^i.p$\}. 
	
	\begin{figure}[h]
		\centering
		\includegraphics[width=6cm]{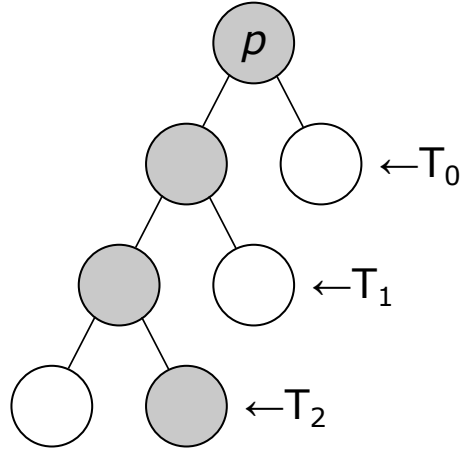}
		\caption{\textit{Golden path} for $n=2$}
	\end{figure}
	
	If $E_n$ has no points on the \textit{golden path}, we add one using the switcher $\sigma$ at $p$ and using $a$ and $b$ push it down the \textit{golden path}. In this case, $g=ba^i\sigma$. Otherwise, we take $g \in \{a,b\}^*$ which pushes the nearest point of the path to $ba^i.p$. Note that $\varphi_i$ is equal to $\frac{1}{2^i}$ on all points of the golden path except $ba^i.p$ where it is $\frac{1}{2^{i+1}}$. In this way, if $E_n$ contains no points in $T_i$, we have that $f(E_n) = \frac{1}{2^i}$, however, with our construction of $g$, $f(gE_n) \le \frac{1}{2^{i+1}}$. Also, if $g \in \{a,b\}^*$ or $g \in \{a,b\}^*\sigma$, as in our case, then $f_k(gE_n) \le f_k(E_n)$. This happens because applying $\sigma$ does not change the value of any $f_k$, and applying $a$ and $b$ can only make it smaller.
	
	Therefore, $\left|\frac{f(gE_n)}{f(E_n)} -1 \right| \ge \frac{1/2^{i+1}}{f(E_n)} \ge \frac{1}{2^{i+2}} \ge \frac{1}{2^n}$, which contradicts the assumption. Hence, $E_n$ must contain at least one point in each of the sets $T_0$, $T_1$, \ldots, $T_{n-2}$ and $|E_n| \ge n-1$. This immediately implies the statement of the lemma.
\end{proof}

Now we shall explicitly construct a sequence $\{E_n\}$ satisfying the condition from Theorem \ref{north-strong}. Our construction will produce a sequence with a stronger property: $f_i(gE_n)=f_i(E_n)$ for any $i$ and for any word $g$ of length $\le n$. This, of course, also means $f(gE_n)=f(E_n)$, which satisfies our requirement.

Take $E_n = \{ a^{-n}b^n.p, a^{-n}b^n.ap, \ldots, a^{-n}b^na^{n-1}.p \}$ . Let $g$ be a word of length $\le n$.

\begin{figure}[h]
	\centering
	\includegraphics[width=7cm]{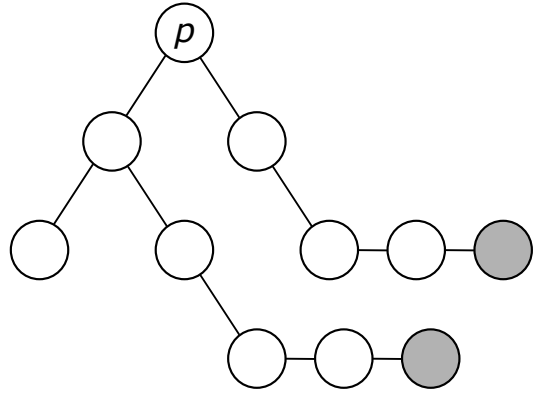}
	\caption{$E_2$ (in grey) as a subset of $X$}
\end{figure}

$E_n$ has exactly one point in each of the sets $T_0$, $T_1$, \ldots, $T_{n-1}$. Each of these points is located in the hair $n$ points away from the binary skeleton. Let $i$ be an index such that $0 \le i \le n-1$. If $g$ contains only shifts, i.e. letters in $\{a,b,a^{-1},b^{-1}\}$ without switches, it cannot change the value of $f_i$ since each $\varphi_i$ is constant on hairs. If the word contains a switch, it might add a point at $p$, but to change the value of $f_i$ this point should move at least $n+1+i$ times -- which is not allowed. Hence, $f_i(gE_n)=f_i(E_n)$ for $i<n$.

\begin{figure}[h]
	\centering
	\includegraphics[width=11cm]{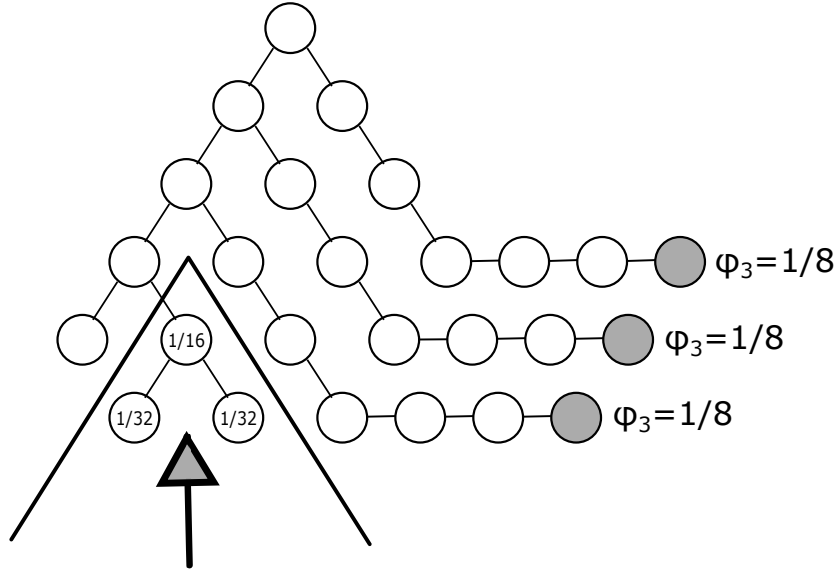}
	\caption{Location of $E_3$ (in grey) and points in $X$ where $\varphi_3< \frac{1}{8}$}
\end{figure}

Now take $i \ge n$. It is easy to see that $f_i(E_n)=\frac{1}{2^i}$ since $E_n$ has no points in $T_i$. In order to get $f_i(gE_n)<\frac{1}{2^i}$, we need $gE_n$ to have at least one point in $T_i$. By definition of $\varphi_i$, it can be done in at least $i+2 > n$ steps. Hence, $f_i(gE_n)=f_i(E_n)$ even when $i \ge n$.

\section{Generalized min-functions}

Denote by $T$ the subspace of sequences in $(0,1]^\mathbb{N}$ whose all but finitely many coordinates are equal to 1. Let $g : T \rightarrow \mathbb{R}^+$ be a function satisfying the following:

\begin{itemize}
	\item $g$ is non-negative and non-decreasing in any variable;	
	\item $g$ is concave, i.e. $\lambda g(u)+(1-\lambda)g(v) \le g(\lambda u + (1-\lambda)v)$ for $0 \le \lambda \le 1$ and $u,v \in T$;
	\item $g$ is symmetric, i.e. $g(x_1, x_2, \ldots, x_n, \ldots) = g(x_{\pi(1)}, \ldots, x_{\pi(n)}, \ldots)$ for any bijection $\pi : \mathbb{N} \rightarrow \mathbb{N}$.
\end{itemize}

Let $\varphi$ be a superharmonic function on $X$ achieving its maximum $1$ at the root point $p$. Let $E = \{q_1, \ldots, q_n\} \in \mathcal{P}_f(X)$ be a finite set. Define $\\ f(E) = g(\varphi(q_1), \ldots, \varphi(q_n), 1, 1, 1, \ldots)$. Then $f$ is non-negative, superharmonic on $\mathcal{P}_f(X)$ and invariant with respect to the switch. 

Alternatively, one could view $g$ as a collection of concave functions $g_n$, $n \ge 0$, where $g_n(x_1, \ldots, x_n) = g(x_1, \ldots, x_n, 1, 1, 1, \ldots)$. In this way, each $g_n$ is a restriction of $g_{n+1}$ where an arbitrary coordinate is taken to be $1$, and $g_n$ are defined by the same properties as $g$.

This gives us a big class of superharmonic functions on the Lamplighter group $\pfxf$. The existence of strong approximations for them is not immediately obvious since there is not much information on what happens if we add or delete a point close to $p$.

An example of such a function is $g(x_1, x_2, \ldots) = \min(h(x_1),h(x_2), \ldots)$, where $h: (0,1] \rightarrow \mathbb{R}$ is a non-negative, non-decreasing and concave function. In particular, this construction produces all min-functions. The theorem below further generalizes this class.


\begin{theorem}\label{ext}
	
	Let $n$ be a positive integer and $r : (0,1]^n \rightarrow \mathbb{R}^+$ be a symmetric non-negative concave function, non-decreasing in each variable. Then the function $g(x_1, \ldots, x_n,\ldots) = r(x_1, \ldots, x_n)$ for $x_1 \le x_2 \le \ldots \le x_n \le \ldots = 1$, extended to $T$ by sorting, satisfies the conditions above. In particular, if $n=1$ and $r(x)=x$ (or, more generally, $n \in \mathbb{N}$ and $r(x_1, \ldots, x_n) = \min(x_1, \ldots, x_n)$), this construction yields the standard min-function. 
\end{theorem}

If $f$ is obtained from some $\varphi$ as described above and $r$ satisfying the theorem above, we call it a \textit{generalized min-function}.

\begin{proof}
	$g$ is non-negative, symmetric and non-decreasing by construction. It remains to check the concavity. We notice that 
	\begin{align*}
		& g(\lambda u + (1-\lambda)v) = g(\lambda u_1+(1-\lambda)v_1, \ldots, \lambda u_n+(1-\lambda)v_n, \ldots) = \\ 
		& = g(\lambda u_{(1)}+(1-\lambda)v_{(1)}, \ldots, \lambda u_{(n)}+(1-\lambda)v_{(n)}, \ldots) = \\
		& = r(\lambda u_{(1)}+(1-\lambda)v_{(1)}, \ldots, \lambda u_{(n)}+(1-\lambda)v_{(n)}) \ge \\ & \ge \lambda r(u_{(1)}, \ldots, u_{(n)}) + (1-\lambda) r(v_{(1)}, \ldots, v_{(n)}) \ge \\
		& \ge \lambda r(u_{((1))}, \ldots, u_{((n))}) + (1-\lambda) r(v_{((1))}, \ldots, v_{((n))}) = \lambda g(u) + (1-\lambda) g(v).	
	\end{align*} 
	
	The indices are sorted so that $u_{((i))}$, $v_{((i))}$ and $\lambda u_{(i)}+(1-\lambda)v_{(i)}$ are all non-decreasing permutations of $u_i$, $v_i$ and $\lambda u_{(i)}+(1-\lambda)v_{(i)}$ respectively. This proves that $g$ is concave.
\end{proof}

\begin{proposition}
	The function $r$ is continuous.
\end{proposition}

\begin{proof}
	Since $0 \le r(x_1, \ldots, x_n) \le r(1, \ldots, 1)$, $r$ is bounded.
	
	Assume the converse. Let $(x_1, \ldots, x_n)$ be a point of discontinuity and put $t = r(x_1, \ldots, x_n)$. Denote $R(\varepsilon)=r(x_1+\varepsilon, \ldots, x_n+\varepsilon)$ for all values $\varepsilon$ where it is well-defined, including some open interval around $0$. The assumption and the fact that $r$ is non-decreasing imply that $R$ is discontinuous at $0$. On the other hand, $R$ is bounded, non-decreasing and concave. Let $R(-a)=t-b \le 0$, where $-a$ is an arbitrary negative number for which $R$ is defined. Now it follows that $t \le R(z) \le t+\frac{zb}{a}$ for $z \ge 0$ and $t-\frac{zb}{a} \le R(z) \le t$ for $z \le 0$, meaning that $R$ is continuous at $0$. Contradiction.
\end{proof}

\begin{theorem}
	If $f$ is a generalized min-function in the sense of Lemma \ref{ext}, then $f$ has an approximation in the strong sense, that is, there is a sequence $\{E_n\}$ such that for any $s \in \pfxf$ we have $\frac{f(sE_n)}{f(E_n)} \rightarrow 1$ as $n \rightarrow \infty$.
\end{theorem}

\begin{proof}
	Let $r = r(x_1, \ldots, x_m)$ and $\varphi$ be the function from the definition of a generalized min-function. As before, without loss of generality we assume $\inf \varphi = 0$.
	
	The properties of $r$ imply that it is also uniformly continuous. By the uniform continuity theorem for metric spaces, we can continuously extend it to a function defined on the closed hypercube $[0,1]^m$, which for convenience we will also call $r$. 
	
	Now, fix $n \in \mathbb{N}$ in order to construct a set $E_n$ satisfying $|\frac{f(sE_n)}{f(E_n)} - 1| < \frac{1}{n}$. The idea is to construct $E_n$ as a set of $m$ points located on hairs away from $p$ and at a distance $>n$ from the binary skeleton. In this case $sE_n$ will consist of $m$ points located on the same hairs and, possibly, some points at a distance $\le n$ from $p$.
	
	Let $\varepsilon = \min\limits_{q \in \mathcal{B}_n} \varphi(q)$, where $\mathcal{B}_n$ is defined as in the proof of Theorem \ref{second}. Then our goal is to take $E_n$ so deep in the tree that $\varphi(q)<\varepsilon$ for all $q \in sE_n$.
	
	There are two possible cases which define the way we are going to construct the approximation. The first case is when $X$ has at least $m$ hairs where the supremum of $\varphi$ is less than $\varepsilon$; the second case is when it has at most $m-1$ hairs with this property. 
	
	
	\textbf{Case 1.} $X$ has at least $m$ hairs where the supremum of $\varphi$ is less than $\varepsilon$.
	
	In this case, we are constructing $E_n$ as a set consisting of $m$ points, one on each hair. The idea is to choose them far enough from the skeleton so the value of $\varphi$ wouldn't change much by perturbing the points.
	
	Denote $y_i = \sup\limits_{x \in H_i} \varphi(x)$, where $H_i$ is the corresponding hair for $i=\overline{1,m}$ satisfying the above property. Let $z_i \in H_i$ be such a point that $\varphi(z_i) \ge \frac{n}{n-1} y_1$. Then it follows from the concavity properties of $\varphi$ on hairs that the set $E_n = \{z_1^{(+n)}, \ldots, z_m^{(+n)}\}$ satisfies $|\frac{f(sE_n)}{f(E_n)} - 1| < \frac{1}{n}$.

	\textbf{Case 2.}  
	We make the following claim which directly follows from the continuity of $r$:
	
	There is a number $\alpha>0$ and a tuple $(y_1, \ldots, y_m)$ such that:
	\begin{itemize}
		\item $0 < y_i < \varepsilon - \alpha$ for all $i$
		\item $\frac{r(y_1+\alpha, \ldots, y_n+\alpha)}{r(y_1, \ldots, y_n)} < 1+\frac{1}{n}$
	\end{itemize}

	If $X$ has only finitely many hairs where the supremum of $\varphi$ is less than $\varepsilon$, we can pick $m$ points $\{ z_1, \ldots, z_m \} \subset X_sk$ such that $\varphi(z_i)<\frac{\alpha}{10n}$ and $\lim\limits_{l \rightarrow \infty} \varphi(z_i^{+l}) > \varepsilon$. According to the proof of Theorem \ref{first}, $0 \le z_i^{(k+1)}-z_i^{(k)} < \frac{\alpha}{3n}$. Hence, we can pick points $ z_1^{(k_1)}, \ldots, z_m^{(k_m)}$ such that $z_i^{(k_i)} \in (y_i+\frac{\alpha}{3} , y_i + \frac{2\alpha}{3})$. After applying the word $s$, $z_i^{(k_i)}$ moves to $z_i^{(k_i+q)}$, with $-n \le q \le n$. This implies that $z_i^{(k_i+q)} \in (y_i, y_i+\alpha)$ and $f(sE_n) \in (r(y_1, \ldots, y_n), r(y_1+\alpha, \ldots, y_n+\alpha))$, however, by construction $f(E_n)$ lies in the same range. As a conclusion, a simple calculation shows that $|\frac{f(sE_n)}{f(E_n)} - 1| < \frac{1}{n}$.
	
\end{proof}

\section{Possible further applications}

The most important open question related to the topic is whether the approximation in a strong sense exists for Green's function $G(E,\varnothing|z)$ for all values of $z$ where it is defined. If Thompson's group is not amenable, we know that:

\begin{itemize}
	\item $G(E,\varnothing|z)$ is defined for some $z=\frac{1}{r}>1$, and 
	\item $G(E,\varnothing|z)$ has no approximation for any $z>1$ as evidenced by the proof of Theorem \ref{north}.
\end{itemize}

If it is amenable, then $G(E,\varnothing|z)$ exists only for $z \le 1$ and, by Northshield's criterion, has an approximation in a strong sense. 

However, the structure of the graph seems to be too complex to explicitly calculate Green's function even in the case $z=1$. Even for graphs with relatively simple structure, for instance, the standard Cayley graph of $\mathbb{Z}^n$ for $n \ge 3$, there is no elementary formula for Green's function.

Another open question is whether such approximations exist for well-defined countable sums of $P^{n_i}f_i$-like expressions from Theorem \ref{fourth}. Since one of the central points of the proof for finite sums was finding the maximum $n_i$, this approach does not work when the sum is infinite. A notable subclass of this class is the set of potentials $Gf(E|z) = \sum\limits_{n=0}^{\infty} P^nf(E) z^n = \sum\limits_F G(E,F|z)f(F) $ for $z<1$ and any bounded positive min-function $f$. They are always well-defined and the existence of $E_n$-approximations in the strong sense for them has not been proved or disproved.


	\section{Free group action is amenable but not extensively amenable}\label{almost-f2}
	
	Consider the following Schreier graph $Z$:
	
	First, we take the right Cayley graph of the free group $\mathbb{F}_2$ generated by $a$ and $b$. Then we cut the edge connecting $e$ and $a$ and replace the part containing $a$ with an infinite tail isomorphic to $\mathbb{Z}^+$ where $a$ acts by moving one edge away from $e$ and $b$ acts trivially.
	
	Define the following superharmonic function $\varphi$ on $Z$:
	
	\begin{itemize}
		\item $\varphi(x) = 1$ if $x$ is on the hair;
		\item $\varphi(x) = 3^{-|x|}$ if $x$ is outside of the hair at a distance $|x|$ from $e$.
	\end{itemize}
	
	\begin{figure}[h]
		\centering
		\includegraphics[width=11cm]{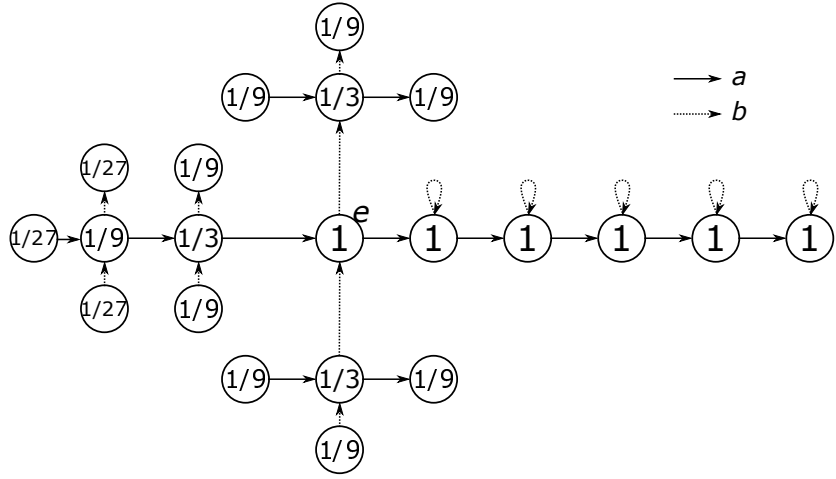}
		\caption {Fragment of graph $Z$ with values of $\varphi$}
	\end{figure}
	
	Now define the corresponding superharmonic min-function $f$ on $\mathcal{P}_f(Z)$:
	
	\begin{itemize}
		\item $f(E) = \min\limits_{x \in E} \varphi(x), \ E \neq \varnothing$;
		\item $f (\varnothing) = 1$.
	\end{itemize}
	
	Then there is no sequence $E_n \in \mathcal{P}_f(X)$ such that for any word $g$ the sequence $\frac{f(gE_n)}{f(E_n)}$ converges to 1.	
	
	\begin{proof}
		Assume such a sequence exists. As before, we can assume $|\frac{f(E_ng)}{f(E_n)}-1|<\frac{1}{n}$ Consider two cases:
		
		$\mathbf{Case \ 1}.$ $\exists N \in \mathbb{N} \ \forall n \ge N: f(E_n)=1.$
		
		This means that starting from some moment all $E_n$ are subsets of the tail. Take $g=pb$, where $p$ is the generator corresponding to the switch at $e$ and $b$ multiplies all elements by $b$ (here we are assuming the right notation, i.e. when applying $g=pb$ we apply $p$ first and then $b$). Then $f(E_ng)=\varphi(b)=\frac{1}{3}$, which contradict the assumption.
		
		$\mathbf{Case \ 2}.$ For infinitely many $n$, $E_n$ is not a subset of the hair, i.e. $f(E_n)<1$.
		
		Let $x_n$ be the element of $E_n$ minimizing $\varphi$ (if there are several, pick any). There are infinitely many $x_n$ terminating in the same letter, say, $b$. Then by taking $g$ to be any letter different from $b^{-1}$ we have that $f(E_ng) = \varphi(xg) = -3^{|x|+1} = \frac{1}{3} \varphi(x) = \frac{1}{3}f(E_n)$ for infinitely many $n$, which prevents the convergence of $\frac{f(gE_n)}{f(E_n)}$ to 1.
		
		In both cases we arrived at a contradiction.
	\end{proof}
	
	We have used Theorem \ref{north} to prove that the action of the free group on $Z$ is not extensively amenable by showing that there is no approximation in a strong sense. However, the action is clearly amenable since we can take arbitrarily long portions of the tail to be the F\o lner sets. This gives us another example of a non-extensively amenable group action.
	
	\section {Some facts about random walks on X}
		Let the Schreier graph $X$ and the point $p$ be defined as above. We can consider a left simple random walk on $X$ given by the measure $\mu(a)=\mu(b)=\mu(a^{-1})=\mu(b^{-1})=\frac{1}{4}$. This random walk is transient and has spectral radius $1$ due to the graph's amenability (more on spectral radii can be found in \cite{woess}). In this section we collect some useful lemmas which we will need in order to show some other results.
		
		\begin{lemma} \label{gpp4}
			The Green's function of the random walk above satisfies $G(p,p|1)=4$.
		\end{lemma}
		
		\begin{proof}
			For any point $x$, denote by $u(x)$ the distance between $p$ and the point of the binary tree closest to $x$ (i.e. $u(x)$ says how many levels down from $p$ the point is located). It is easy to check that the function $f(x)=2^{2-u(x)}$ satisfies $f-Pf=\delta_p$ and converges to $0$ on any Poisson boundary of the tree. Thus, the harmonic component in its Riesz decomposition (see Section \ref{potentials} for definitions) is $0$ and $f$ is a potential. From the fact that $f-Pf=\delta_p$ it follows that this potential is equal to Green's function $G(\cdot,p)$. It remains to notice that $f(p)=2^2=4$.
		\end{proof}

		\begin{lemma} \label{walkx1}
			Let $q_n$ be a simple random walk on $X$ as defined above starting at a point $q_0$ and let $\varphi$ be a bounded positive superharmonic function on $X$. Then there is a non-zero probability that $\lim\limits_{n \rightarrow \infty} \varphi(q_n) \le \varphi(q_0)$. 
		\end{lemma}
		
		\begin{proof}
			
			Since $\varphi$ is superharmonic, the sequence $\varphi(q_n)$ can be seen as a supermartingale. Hence, the expectation $\mathbb{E} \varphi(q_n)$ is non-increasing and $\mathbb{E} \varphi(q_n) \le \varphi(q_0)$ for all $n$. Then $\lim \limits_{n \rightarrow \infty}\varphi(q_n)$ exists almost surely by the supermartingale convergence theorem and $\mathbb{E}(\lim \limits_{n \rightarrow \infty}\varphi(q_n)) \le \varphi(q_0)$ by Fatou's lemma. This immediately implies the statement.
		\end{proof}
		
		\begin{lemma} \label{walkx2}
			Let $q_n$ be a simple random walk on $X$ as defined above such that $q_0=p$ and let $\varphi$ be a positive superharmonic function on $X$ such that $\varphi(q) \le t = \varphi(p) $ for any $q \in X$, and $\inf\limits_{x \in X} \varphi (x) = 0$. Then for any $\varepsilon > 0$  there is a non-zero probability that $\lim\limits_{n \rightarrow \infty} \varphi(q_n) < \varepsilon$. 
		\end{lemma}
		\begin{proof}
			
			Let $r$ be a point on the binary skeleton of $X$ with $\varphi(r) < \varepsilon$ (it exists by the assumption on $\varphi$). Denote by $X_r$ the binary tree rooted in $r$. There is a non-zero probability that for some number $N$ we have that $q_n \in X_r, \ n \ge N$ and $q_N = r$. Applying the proof of Lemma \ref{walkx1} to $X_r$ and $\tilde{q}_n=q_{N+n}$, we get that $\lim \varphi(q_n) \le \varphi(r) < \varepsilon $ with positive probability.
		\end{proof}

	\section{Min-functions are potentials}\label{potentials}
	
	\begin{definition}
		Suppose that $X$ is a transient graph with a measure defining the Markov operator $P$. For a function $f : X \rightarrow \mathbb{R}$, we define its potential to be $g(x) = Gf(x) = \sum\limits_y G(x,y)f(y)$, where $G$ is the corresponding Green's function. 
	\end{definition}
	
	It was proven in \cite{woess} (Chapter 24) that any positive superharmonic function can be decomposed into the sum of a potential $g=Gf$ and a harmonic function $h \ge 0$. We prove that min-functions, and hence their finite and well-defined countable sums, are potentials.
	
	Another way to look at the potential is in terms of $P$: $g(x)=Gf(x) = \sum\limits_{n=0}^{\infty} P^n f(x)$. The definitions are clearly equivalent due to the definitions of the Green's function and Markov operator.
	
	\begin{theorem}
		Let $X$ be the Schreier graph of Thompson's group $F$ with $p$ being the root vertex. Let $\varphi$ be a positive superharmonic function on $X$ such that $\varphi(q) \le \varphi(p) $ for any $q \in X$ and $\inf\limits_{x \in X} \varphi (x) = 0$. As before, define the min-function $f : \mathcal{P}_f(X) \rightarrow \mathbb{R} $ by 
		\begin{itemize}
			\item $f(E) = \min\limits_{x \in E} \varphi(x), \ E \neq \varnothing$ 
			\item $f (\varnothing) = \varphi (p)$.
		\end{itemize}

		Then $f$ is a potential.
		
	\end{theorem}
	
	\begin{proof}
		
		We will prove the theorem by simulating a random walk $E_n \in \mathcal{P}_f(X)$ starting at an arbitrary subset $E_0$. The idea is to prove that $\lim\limits_{n \rightarrow \infty} f(E_n) = 0$ almost surely.

		We can regard the sequence $E_n$ from the following prospective: each $E_n$ is a subset of points (\objs) of $X$. Each time we make a move from $E_n$ to $E_{n+1}$, we either move each $\obj$  by a letter ($a$, $b$, $a^{-1}$ or $b^{-1}$), add an $\obj$ at $p$ to the set (if there is none) or remove an $\obj$ from $p$ from it (if there is one). 
		
		According to the Borel-Cantelli lemma, since the probability of invoking the generator $\sigma$ in one move is $\frac{1}{5}$ and we make infinitely many moves, with probability $1$ we will add or remove a new $\obj$ infinitely many times. Since initially we have finitely many $\objs$ and we cannot remove an $\obj$ without adding it first, this implies that as we go onwards from $E_0$ we will almost surely add a new $\obj$ infinitely many times. 
		
		The trajectory of any $\obj$ can be described as a (possibly terminated) simple random walk on the graph $X$. By Lemma \ref{gpp4} its Green's function $G$ satisfies $G(p,p)=4$, meaning that a random walk starting at $p$ visits $p$ four times on average, including the starting position. From Lemma 1.13 (a) in \cite{woess} it follows that an $\obj$ has a $\frac{1}{4}$ chance of never returning back to $p$. Applying the Borel-Cantelli argument again, we see that an infinite number of $\objs$ will never be removed.
		 
		Since each hair of $X$ is recurrent, any $\obj$ will converge to some end of the binary skeleton. Lemma \ref{walkx2} together with the Borel-Cantelli lemma (again!) show that for any $\varepsilon > 0$ sooner or later some $\obj$ will converge to an end of the graph such that its trajectory $\{q_n\}$ will satisfy $\lim\limits_{n \rightarrow \infty} \varphi(q_n) < \varepsilon$. In this case, $\lim\limits_{n \rightarrow \infty} f(E_n) < \varepsilon$ and, since  $\varepsilon$ is arbitrary, $\lim\limits_{n \rightarrow \infty} f(E_n) = 0$.

		Now assume $f$ is not a potential. Then let $f(E)=g(E)+h(E)$ be the Riesz decomposition of $f$, where $g$ is a potential and $h$ is a positive bounded harmonic function. Since $g$ is non-negative, by the squeeze theorem $\lim\limits_{n \rightarrow \infty} h(E_n) = 0$. On the other hand, by harmonicity of $h$, $\mathbb{E} h(E_n) = h(E_0)>0 $ and by dominated convergence theorem $\lim\limits_{n \rightarrow \infty} h(E_n)$ exists and satisfies $\mathbb{E}(\lim\limits_{n \rightarrow \infty} h(E_n)) = \lim\limits_{n \rightarrow \infty} \mathbb{E}h(E_n) = h(E_0) > 0$. Contradiction.		
	\end{proof}


\begin{thebibliography}{AA99}
	
	
	
	\bibitem{woess}\textsc{Wolfgang Woess}, {\it Random Walks on Infinite Graphs and Groups (Cambridge Tracts in Mathematics)} 
	
	
		\bibitem{burillo}\textsc{Jos\'e Burillo}, {\it Introduction to Thompson's group F} \hfil {\tt https://mat-web.upc.edu/people/pep.burillo/F\%20book.pdf}
		
		\bibitem{cfp}\textsc{J.W. Cannon, W.J. Floyd, W.R. Parry}, {\it Introductory notes on Richard Thompson's groups} \hfil {\tt http://people.math.binghamton.edu/matt/thompson/cfp.pdf}
	
		\bibitem{yeow} \textsc{Daniel Yeow}, {\it Introduction to Thompson's group F (Honours Thesis)} \hfil {\tt http://www.danielyeow.com/wp-content/uploads/2009/06/ honoursthesisfinal.pdf}
		
	
	\bibitem{bcw}\textsc{J. Burillo, S. Cleary, B. Wiest } {\it Computational explorations in Thompson's group F}{\tt https://arxiv.org/pdf/math/0506346.pdf}	
	
	\bibitem{akhmedov}\textsc{Azer Akhmedov}, {\it Non-amenability of R.Thompson's group F} {\tt https://arxiv.org/abs/0902.3849}
	
	\bibitem{shavg}\textsc{E.T.Shavgulidze}, {\it About amenability of subgroups of the group of diffeomorphisms of the interval} {\tt https://arxiv.org/abs/0906.0107}
	
	
	\bibitem{JMMS}\textsc{Juschenko, K., Matte Bon, N., Monod, N., de la Salle, M.}, {\it Extensive amenability and an application to interval exchanges.} arXiv preprint arXiv:1503.04977.
	
	\bibitem{JdlS}\textsc{Juschenko, K., de la Salle, M.}, {\it Invariant means for the wobbling group}.
	
	\bibitem{J-book}\textsc{Juschenko, K.}, {\it Amenability}. Book in preparation. \hfil {\tt http://www.math.northwestern.edu/\string~juschenk/book.html}.
	
	
	\bibitem{northshield}\textsc{Sam Northshield}, {\it Amenability and Superharmonic Functions} \hfil {\tt https://digitalcommons.plattsburgh.edu/mathematics\_facpubs/19/}

	
	\bibitem{notstram}\textsc{Y. Hartman, K. Juschenko, O. Tamuz, P. V. Ferdowsi}. {\it Thompson's group F is not strongly amenable} \hfil {\tt https://arxiv.org/abs/1607.04915}
	
	\bibitem{savchuk}\textsc{Dmytro Savchuk}. {\it Some graphs related to Thompson's group F} \hfil {\tt https://arxiv.org/abs/0803.0043}
	

	

\end{thebibliography}
\end{document}